\documentclass[11pt]{amsart}%
\usepackage{graphicx}
\usepackage{amscd}
\usepackage{amsmath}
\usepackage{mathtools}
\usepackage{amsfonts}
\usepackage{enumerate}
\usepackage{amssymb}
\usepackage[pagebackref,hypertexnames=false, colorlinks, citecolor=blue, urlcolor=cyan, linkcolor=red]{hyperref}

\usepackage{color}%
\providecommand{\U}[1]{\protect\rule{.1in}{.1in}}

\newtheorem{theorem}{Theorem}
\theoremstyle{plain}

\newtheorem{corollary}{Corollary}[section]

\newtheorem{lemma}{Lemma}[section]

\numberwithin{equation}{section}
\numberwithin{theorem}{section}

\newcommand{\sgn}{\text{sgn}}
\everymath{\displaystyle}
\begin{document}
\title[Spectral Analysis for the Exceptional $X_m$-Jacobi Equation]{Spectral Analysis for the Exceptional $X_m$-Jacobi Equation}
\author{Constanze Liaw}
\address{Department of Mathematics, Baylor University, One Bear Place \#97328, Waco, TX 76798-7328}
\email{Constanze\_Liaw@baylor.edu}
\author{Lance Littlejohn}
\address{Department of Mathematics, Baylor University, One Bear Place \#97328, Waco, TX 76798-7328}
\email{Lance\_Littlejohn@baylor.edu}
\author{Jessica Stewart}
\address{Department of Mathematics and Computer Science, Goucher College, 1024 Dulaney Valley Road, Baltimore, MD 21204}
\email{Jessica\_Stewart@baylor.edu}
\maketitle
	
\begin{abstract}
We provide the mathematical foundation for the $X_m$-Jacobi spectral theory.
Namely, we present a self-adjoint operator associated to the differential expression with the exceptional $X_m$-Jacobi orthogonal polynomials as eigenfunctions. This proves that those polynomials are indeed eigenfunctions of the self-adjoint operator (rather than just formal eigenfunctions).
Further, we prove the completeness of the exceptional $X_m$-Jacobi orthogonal polynomials (of degrees $m, m+1, m+2, \hdots$) in the Lebesgue--Hilbert space with the appropriate weight. In particular, the self-adjoint operator has no other spectrum.
\end{abstract}

	
\section{Introduction}
The classical orthogonal polynomials of Laguerre, Jacobi, and Hermite are the foundational examples of orthogonal polynomial theory.  As shown by Routh in 1884 \cite{Routh}, but most often attributed to Bochner in 1929 \cite{Bochner}, these three families of polynomials are, up to affine transformation of $x$,  the only polynomial sequences satisfying the following two conditions: First, they contain an infinite sequence of polynomials $\left\{p_n\right\}_{n=0}^\infty$, where $p_n$ has degree $n$, such that for each $n\in\mathbb N_0$, $y=p_n$ satisfies a second order eigenvalue equation of the form
\[p(x)y''+q(x)y'+r(x)y=\lambda y\,,\]
where the polynomials $p(x)$, $q(x)$, and $r(x)$ are determined by the corresponding differential expression (Laguerre, Jacobi or Hermite). Second, each of the eigenpolynomials is orthogonal in a weighted $L^2$ space where the associated weight has finite moments.

In recent years, there has been interest in the area of \textit{exceptional} orthogonal polynomials, which presents a way to generalize Bochner's classification theorem.  The most striking difference between classical orthogonal polynomials and their exceptional counterparts is that the exceptional sequences allow for gaps in the degrees of the polynomials.  We denote an exceptional orthogonal polynomial sequence $\left\{p_{m,n}\right\}_{n\in \mathbb N_0\diagdown A}$ by using ``$X_m$'', where the subscript $m=\left|A\right|$ denotes the number of gaps (or the \textit{codimension} of the sequence).  We require that the associated second order differential expression preserve the space spanned by the exceptional polynomials, but no space with smaller codimension.  Consequently, the coefficients of the second order differential equation are not necessarily polynomial. Remarkably, despite removing any finite number of polynomials, the sequences remain complete in their associated space.

Research in the area of exceptional orthogonal polynomials did not develop from a desire to generalize Bochner's theorem; rather, the exceptional polynomials were discovered in the context of quantum mechanics where researchers were looking for a new approach, outside of the classical Lie algebraic \cite{Gonzalez, Kamran-Olver,Quesne} setting, to solving spectral problems for second order linear differential operators with polynomial eigenfunctions.  In particular, they were discovered in \cite{KMUG, KMUG1} while developing a direct approach \cite{KMUG8} to exact or quasi-exact solvability for spectral problems. The first examples of these exceptional polynomials were introduced in 2009 by G\'{o}mez-Ullate, Kamran and Milson \cite{KMUG, KMUG1}, who completely characterized all $X_1$-polynomial sequences. Their result showed that the only polynomial families of codimension one (in particular, having no solution of degree zero) satisfying a second order eigenvalue problem are the $X_1$-Jacobi and $X_1$-Laguerre polynomials. Explicit examples of the $X_2$ families were given by Quesne \cite{Quesne, Quesne2}, who used the Darboux transformation and shape invariant potentials to find these new families. 

Higher-codimensional families, including the $X_m$-Laguerre and $X_m$-Jacobi exceptional polynomial sequences, were first observed by Odake and Sasaki \cite{Odake-Sasaki1}.  Further generalizations were observed regarding two distinct types of $X_m$-Laguerre polynomials by G\'{o}mez-Ullate, Kamran and Milson \cite{KMUG5, Zeros}.  These $X_m$-Laguerre polynomial families do not contain polynomials of degree $n\in \mathbb{N}$ for $0\leq n\leq m-1$. Furthermore, Liaw, Littlejohn, Milson, and Stewart \cite{Liaw-Littlejohn-Milson-Stewart} show the existence of a third type of $X_m$-Laguerre polynomials.  The Type III $X_m$-Laguerre polynomial sequence omits polynomials of degree $n\in \mathbb{N}_0$ for $1\leq n\leq m$.  This new class of polynomials can be derived from the quasi-rational eigenfunctions of the classical Laguerre differential expression by Darboux transform as well as a gauge transformation of the Type I exceptional $X_m$-Laguerre expression.  

Following the discovery of exceptional polynomials, there has been a desire to study the properties of these polynomials more rigorously.  The explanation for existence via the Darboux transformation of the higher-codimension $X_m$-Jacobi and $X_m$-Laguerre polynomials and a remarkable observation regarding the completeness of the $X_m$-polynomial families was given by G\'{o}mez-Ullate, Kamran and Milson \cite{KMUG2}.  G\'{o}mez-Ullate, Marcell\'{a}n, and Milson studied the interlacing properties of the zeros for both the exceptional Jacobi and exceptional Type I and Type II Laguerre polynomials along with their asymptotic behavior \cite{Zeros}.  The properties of the Type III $X_m$-Laguerre polynomials is studied \cite{Liaw-Littlejohn-Milson-Stewart}. 

Further studies regard the spectral analysis for the polynomials.  The spectral analysis for the $X_1$-Jacobi polynomials (for $m=1$, $A = \{0\}$) may be found in \cite{Liaw-Littlejohn-Stewart} along with an analysis of properties resulting from an extreme parameter choice, and for the $X_1$-Laguerre polynomials, the spectral analysis was completed in \cite{Atia-Littlejohn-Stewart}.  For all three types of the $X_m$-Laguerre polynomials, a complete spectral study is completed in \cite{Liaw-Littlejohn-Milson-Stewart}.

In Section \ref{s-Jacobi} we introduce the $X_m$-Jacobi differential expression along with some properties. In Section \ref{s-SA} we then apply the Glazman-Krein-Naimark theory to obtain a self-adjoint operator associated to the differential expression, for which the corresponding domain contains the exceptional $X_m$-Jacobi orthogonal polynomials (see Theorem \ref{t-SA-operator}). Further, we show the completeness of the exceptional $X_m$-Jacobi orthogonal polynomials (of degrees $m, m+1, m+2, \hdots$) in the Lebesgue--Hilbert space with the appropriate weight (see Theorem \ref{t-completeness}). Summing up, we present the spectral analysis of the $X_m$-Jacobi differential expression.

\section{Some Properties of the Exceptional $\text{X}_m$-Jacobi Expression}\label{s-Jacobi}
We begin by summarizing some properties of the exceptional $\text{X}_m$-Jacobi expression as described in \cite{KMUG,Zeros}. The parameters $\alpha$ and $\beta$ are assumed to satisfy 
\begin{align}\label{e-conditions}
\alpha, \beta>-1\,,\quad \alpha+1-m-\beta\notin \left\{0,1,\ldots,m-1\right\}\,,\quad\text{and}\quad \sgn(\alpha+1-m)=\sgn\beta
\end{align}
in accordance with \cite[Section 5.2]{Zeros}, unless otherwise noted. 

The exceptional $\text{X}_m$-Jacobi polynomial of degree $n\geq m$, $P_{m,n}^{(\alpha,\beta)}$ is  given in terms of the classical Jacobi polynomials $\left\{P_{k}^{(\alpha,\beta)}\right\}_{k=0}^\infty$ by
\begin{align*}
P_{m,n}^{(\alpha,\beta)}(x)
&=
\frac{(-1)^m}{\alpha+1+n-m}
\left[
\frac{1}{2} ( \alpha+\beta +n-m+1)(x-1)P_{m}^{(-\alpha-1,\beta-1)}(x)P_{n-m-1}^{(\alpha+2,\beta)}(x)\right.\\
&\qquad\qquad\qquad\qquad+
\left.(\alpha -m +1)P_{m}^{(-\alpha-2,\beta)}(x)P_{n-m}^{(\alpha+1,\beta-1)}(x) \right].
\end{align*}
The exceptional $\text{X}_m$-Jacobi polynomials satisfy the second-order differential equation $$T_{\alpha, \beta, m}[y](x)=\lambda_n y(x)$$ for $x\in (-1,1)$, where the exceptional $X_m$-Jacobi differential expression is given by
\begin{align}
T_{\alpha,\beta, m}[y](x)&:=(1-x^2)y''(x)+\left(\beta-\alpha-(\beta+\alpha+2)x-2(1-x^2)\left(\log(P_m^{(-\alpha-1,\beta-1)}(x))\right)'\right)y'(x) \nonumber \\
&\quad +\left((\alpha-\beta-m+1)m-2\beta(1-x)\left(\log(P_m^{(-\alpha-1,\beta-1)}(x))\right)'\right)y(x) \label{Differential Expression}
\end{align}
and $\lambda_n=-(n-m)(1+\alpha+\beta+n-m)$.

In Lagrangian symmetric form, the exceptional $\text{X}_m$-Jacobi differential expression \eqref{Differential Expression} writes
\begin{align*}
T_{\alpha,\beta, m}[y](x)&=\frac{1}{W_{\alpha,\beta, m}(x)}\left[\left(W_{\alpha,\beta, m}(x)(1-x^2)y'(x)\right)'\right.\\
& \qquad\qquad\left.+W_{\alpha,\beta, m}(x)\left(m(\alpha-\beta-m+1)-2\beta(1-x)\left(\log(P_m^{(-\alpha-1,\beta-1)}(x))\right)'\right)y(x)\right]
\end{align*} 
for $x\in (-1,1),$ where $W_{\alpha,\beta,m}$ is the exceptional $\text{X}_m$-Jacobi weight function given by 
\begin{equation*}
W_{\alpha,\beta, m}(x)=\frac{(1-x)^\alpha(1+x)^\beta}{\left(P_m^{(-\alpha-1,\beta-1)}(x)\right)^2} \quad \mbox{ for } x\in(-1,1)\,. 
\end{equation*}
The restrictions on $\alpha$ and $\beta$ ensure that $W_{\alpha,\beta,m}(x)$ has no singularities for $x\in [-1,1]$ and consequently, all moments are finite.

The exceptional $X_m$-Jacobi polynomials $\left\{P_{m,n}^{(\alpha,\beta)}\right\}_{n=m}^\infty$ are orthogonal with respect to the weight function $W_{\alpha,\beta, m}(x)$.

The eigenvalue equation $T_{\alpha,\beta, m}[y] = \lambda y$ does not have any polynomial solutions of degree $n$ for $0\leq n \le m-1$. Despite this fact, it is interesting that the exceptional $\text{X}_m$-Jacobi polynomials $\left\{P_{m,n}^{(\alpha,\beta)}\right\}_{n=m}^\infty$ form a complete sequence in the Hilbert-Lebesgue space $L^2((-1,1);W_{\alpha,\beta, m})$, defined by
\[L^2((-1,1);W_{\alpha,\beta, m}):=\left\{f:(-1,1)\rightarrow \mathbb C:f \text{ is measurable and } \int_{-1}^1\left|f\right|^2W_{\alpha,\beta,m}< \infty\right\}\,.
\]

\section{Exceptional $\text{X}_m$-Jacobi Spectral Analysis}\label{s-SA}
We follow the methods outlined in the classical texts of Akhiezer and Glazman \cite{Glazman}, Hellwig \cite{Hellwig}, and Naimark \cite{Naimark}.

The maximal domain associated with $T_{\alpha,\beta, m}[\cdot]$ in the Hilbert space $L^2((-1,1),W_{\alpha,\beta, m})$ is:
\begin{equation}\label{Maximal Domain}
\Delta=\left\{f:(-1,1)\rightarrow \mathbb C\big|f,f'\in AC_{loc}(-1,1); f, T_{\alpha,\beta,m}[f]\in L^2((-1,1),W_{\alpha,\beta, m})\right\}\,.
\end{equation}  The maximal domain $\Delta$ is the largest subspace of functions of $L^2((-1,1); W_{\alpha,\beta, m})$ for which $T_{\alpha, \beta,m}$ maps into $L^2((-1,1); W_{\alpha,\beta, m})$. The associated maximal operator is
\begin{equation*}
S^1_{\alpha,\beta,m}:\mathcal D(S^1_{\alpha,\beta,m})\subset L^2((-1,1),W_{\alpha,\beta, m})\rightarrow L^2((-1,1),W_{\alpha,\beta, m})\,
\end{equation*} where $S^1_{\alpha,\beta,m}$ is defined by
\begin{align}
S^1_{\alpha,\beta,m}[f]&:= T_{\alpha,\beta,m}[f] \label{Maximal Operator}\\
f\in \mathcal{D}(S^1_{\alpha,\beta,m})&:=\Delta\,. \nonumber
\end{align}

For $f, g \in \Delta$, Green's Formula may be written as
\begin{equation}\label{Green's Formula}
\int_{-1}^1T_{\alpha,\beta,m}[f](x)\overline{g}(x)W_{\alpha,\beta, m}(x)\,dx=[f,g](x)\big|_{-1}^1+\int_{-1}^1f(x)T_{\alpha,\beta,m}[\overline{g}](x)W_{\alpha,\beta, m}(x)\,dx
\end{equation}
where $[\cdot,\cdot](\cdot)$ is the sesquilinear form defined by:
\begin{align}
[f,g](x)&=W_{\alpha,\beta, m}(x)(1-x^2)(f'(x)\overline{g}(x)-f(x)\overline{g}'(x))\nonumber \\
&=\frac{(1-x)^{\alpha+1}(1+x)^{\beta+1}}{\left(P_m^{(-\alpha-1,\beta-1)}(x)\right)^2}(f'(x)\overline{g}(x)-f(x)\overline{g}'(x)) \quad (x\in (-1,1))\,\label{Sesquilinear Form}
\end{align} and where 
\[
[f,g](x)\mid_{x=-1}^{x=1}:=[f,g](1)-[f,g](-1)\,.
\] By the definition of $\Delta$ and the classical H\"older's inequality, notice that the limits
\[\left[f,g\right](-1):=\lim_{x\rightarrow -1^+}\left[f,g\right](x)\quad \text{and} \quad
\left[f,g\right](1):=\lim_{x\rightarrow 1^-}\left[f,g\right](x)\]
exist and are finite for each $f,g \in \Delta$.

The adjoint of the maximal operator in $L^2((-1,1); W_{\alpha,\beta,m})$ is the minimal operator,
\[
S^0_{\alpha,\beta,m}:\mathcal D(S^0_{\alpha,\beta,m})\subset L^2((-1,1),W_{\alpha,\beta, m})\rightarrow L^2((-1,1),W_{\alpha,\beta, m})\,
\] where $S^0_{\alpha,\beta,m}$ is defined by
\begin{align*}
S^0_{\alpha,\beta,m}[f]&:= T_{\alpha,\beta,m}[f]\\
f\in \mathcal{D}(S^0_{\alpha,\beta,m})&:=\left\{f\in \Delta\big|\left[f,g\right]\big|_{-1}^1=0 \mbox{ for all } g\in \Delta\right\}\,.
\end{align*}

We seek to find a self-adjoint extension $S_{\alpha,\beta, m}$ in $L^2((-1,1); W_{\alpha,\beta, m})$ generated by $T_{\alpha,\beta, m}[\cdot]$, which has the exceptional $\text{X}_m$-Jacobi polynomials $\left\{P_{m,n}^{(\alpha,\beta)}\right\}_{n=m}^\infty$ as eigenfunctions.  In order to achieve this goal, we need to study the behavior of solutions at the singular endpoints $x=-1$ and $x=1$ so as to determine the deficiency indices and find the appropriate boundary conditions (if any).

First, we obtain the deficiency indices via Frobenius Analysis. They depend on the values of the parameters $\alpha$ and $\beta$.

The endpoints $x=-1$ and $x=1$ are, in the sense of Frobenius, regular singular endpoints of the differential expression $T_{\alpha,\beta, m}[\cdot]=0$.  We first apply Frobenius analysis to the endpoint $x=1$.  By multiplying the exceptional $\text{X}_m$-Jacobi expression $T_{\alpha,\beta, m}[y]$ by $\frac{x-1}{x+1}$, we obtain
\[
\left(\frac{x-1}{x+1}\right)\left(T_{\alpha,\beta, m}[y](x)-\lambda_n y(x)\right)=(x-1)^2y''(x)-(x-1)p(x)y'(x)+q(x)y(x)
\]
with \[p(x)=\frac{\beta-\alpha-(\alpha+\beta+2)x}{x+1}-2\left(\log(P_m^{(-\alpha-1,\beta-1)})\right)'(x-1)\]
and \[q(x)=\left(\frac{x-1}{x+1}\right)\left(-(\alpha-\beta-m+1)m-2\beta\left(\log(P_m^{(-\alpha-1,\beta-1)})\right)'(x-1)\right)\,.\]

Evaluating the above equation at $x=1$ yields the indicial equation
\[0=r(r-1)-rp(1)+q(1)=r(r+\alpha)\,.\]
Therefore, two linearly independent solutions to $T_{\alpha,\beta, m}[y]-\lambda_ny=0$ behave asymptotically (near $x=1$, e.g. on the interval $(0,1)$) like 
\[
z_1(x)=1\quad \mbox{ and }\quad z_2(x)=(x-1)^{-\alpha}
\] near $x=1$.

For all allowable values of $\alpha$ and $\beta$, 
\[
\int_0^1\left|z_1(x)\right|^2W_{\alpha,\beta, m}(x)\,dx<\infty\,;
\]
while 
\[
\int_0^1\left|z_2(x)\right|^2W_{\alpha,\beta, m}(x)\,dx<\infty\,
\] only for $-1<\alpha<1$.

In a similar manner, multiplying the exceptional $\text{X}_m$-Jacobi expression $T_{\alpha,\beta, m}[y]-\lambda_{\alpha,\beta,m}y$ by $(x+1)/(x-1)$, results in an indicial equation
\[
r(r+\beta)=0;
\] and two linearly independent solutions will behave (asymptotically) like 
\[
y_1(x)=1\quad \mbox{ and }\quad y_2(x)=(x+1)^{-\beta}
\] near $x=-1$.

For all allowable values of $\alpha$ and $\beta$, 
\[
\int_{-1}^0\left|y_1(x)\right|^2W_{\alpha,\beta, m}(x)\,dx<\infty\,;
\]
while 
\[
\int_{-1}^0\left|y_2(x)\right|^2W_{\alpha,\beta, m}(x)\,dx<\infty\,
\] only for $-1<\beta<1$.

As a consequence, we have the following results.

\begin{theorem}\label{Frobenius Analysis}
Let $T_{\alpha,\beta, m}[y]-\lambda_{\alpha,\beta,m}$ be the exceptional $\text{X}_m$-Jacobi differential expression \eqref{Differential Expression} on the interval $(-1,1)$.
\begin{enumerate}
 \item{$T_{\alpha,\beta, m}[\cdot]$ is in the limit-point case at $x=-1$ for $\beta\geq 1$ and limit-circle for $-1<\beta<1$.}
  \item{$T_{\alpha,\beta, m}[\cdot]$ is in the limit-point case at $x=1$ for $\alpha\geq 1$ and limit-circle for $-1<\alpha<1$.}
\end{enumerate}
\end{theorem}

\begin{corollary}\label{Deficiency Indices}
The minimal operator $S^0_{\alpha, \beta, m}$ in $L^2((-1,1),W_{\alpha,\beta, m})$ has the following deficiency indices:
\begin{enumerate}
\item{For $\alpha, \beta\geq 1$, $S^0_{\alpha, \beta, m}$ has deficiency index $(0,0)$.}
\item{For $\alpha \geq 1$ and $0<\beta<1$, $S^0_{\alpha, \beta, m}$ has deficiency index $(1,1)$.}
\item{Similarly, for $\beta \geq 1$ and $0<\alpha<1$, $S^0_{\alpha, \beta, m}$ has deficiency index $(1,1)$.}
\item{For $-1<\alpha,\beta<0$, $S^0_{\alpha, \beta, m}$ has deficiency index $(2,2)$.}
\end{enumerate}
\end{corollary}

Next we formulate the self-adjoint operators.

\begin{theorem}\label{t-SA-operator}
The self-adjoint operator $S_{\alpha,\beta,m}$ in
$L^{2}((-1,1);W_{\alpha,\beta, m}),$ generated by the exceptional
$X_m$-Jacobi differential expression $T_{\alpha, \beta,m}$ 
is 
given by%
\begin{align*}
S_{\alpha,\beta,m}[f]  &  =T_{\alpha, \beta,m}[f]\\
f  &  \in\mathcal{D}(S_{\alpha,\beta,m}),
\end{align*}
where
\begin{equation}\label{Boundary Conditions}
\mathcal{D}(S_{\alpha,\beta,m})=\left\{
\begin{array}
[c]{ll}%
\Delta & \text{if }\alpha\geq 1\text{ and }\beta\geq 1\medskip\\
\{f\in\Delta\mid\lim\limits_{x\rightarrow-1^{+}}(1+x)^{\beta+1}f^{\prime
}(x)=0\} & \text{if }\alpha\geq 1\text{ and }0<\beta<1\medskip\\
\{f\in\Delta\mid\lim\limits_{x\rightarrow 1^{-}}(1-x)^{\alpha+1}f^{\prime
}(x)=0\} & \text{if }0<\alpha<1\text{ and }\beta\geq 1\medskip\\%
\begin{array}
[c]{l}%
\!\!\!\{f\in\Delta\mid\lim\limits_{x\rightarrow 1^{-}}(1-x)^{\alpha
+1}f^{\prime}(x)\\
\qquad\qquad=\lim\limits_{x\rightarrow-1^{+}}(1+x)^{\beta+1}f^{\prime}(x)=0\}
\end{array}
&
\begin{array}
[c]{l}%
\!\!\text{for all other choices of parameters}\\
\!\!\text{that are allowed by \eqref{e-conditions}.}
\end{array}
\end{array}
\right. %
\end{equation}
\end{theorem}


\begin{proof}
If the parameters satisfy $\alpha, \beta \geq 1$, then there is only one self adjoint extension (restriction) of the minimal operator $S_{\alpha,\beta,m}^0$ (maximal operator $S_{\alpha,\beta,m}^1$); that is, the maximal and minimal operator coincide and $S_{\alpha,\beta, m}=S_{\alpha,\beta,m}^0=S_{\alpha,\beta, m}^1$.

Suppose that $0<\alpha< 1$ and $\beta \geq 1$, then there are infinitely many self-adjoint extensions of the minimal operator $S_{\alpha,\beta,m}^0$.  From Corollary \ref{Deficiency Indices}, the deficiency index equals $(1,1)$, which means that $\mathcal D(S^0_{\alpha,\beta,m})$ is a subspace of codimension 2 in $\Delta$.  We will restrict the maximal domain $\Delta$ by imposing a suitable boundary condition which is invoked by the sesquilinear form $[\cdot,\cdot](\cdot)$ defined by \eqref{Sesquilinear Form}.  First note that $h(x)=(1-x)^{-\alpha}\in \Delta$ since
\[
T_{\alpha,\beta, m}[(1-x)^{-\alpha}]=\mathcal O ((1-x)^{-\alpha}) \quad \mbox{(near }x=1\mbox{)},
\] which implies $T_{\alpha,\beta, m}[g]\in L^2((-1,1);W_{\alpha,\beta, m})$ because $\alpha<1$.  Further, the constant function satisfies $1\in \Delta$ and 
\begin{equation}\label{label}
\left[h, 1\right]\mid_{x=-1}^{x=1}=\left[h,1\right](1)=-\dfrac{\alpha\, 2^{\beta+1}}{\left(P_{m}^{(-\alpha-1,\beta-1)}(1)\right)^2} = \alpha \,2^{\beta+1}\neq 0\,,
\end{equation}
where we used standard identities for the Jacobi polynomials and the Gamma function to find
\[
P_{m}^{(-\alpha-1,\beta-1)}(1) = \frac{\Gamma(-\alpha + m)}{m!\,\Gamma(\beta+ m-\alpha-1)}\frac{\Gamma(\beta+ m-\alpha-1)}{\Gamma(-\alpha)} = \frac{\Gamma(-\alpha + m)}{m!\,\Gamma(-\alpha)}
= \frac{m!\,\Gamma(-\alpha)}{m!\,\Gamma(-\alpha)} = 1.
\]
In particular, we obtain from equation \eqref{label} that the constant function $1\notin \mathcal D(S^0_{\alpha,\beta,m}).$

For $0<\beta< 1$ and $\alpha \geq 1$, we can prove the corresponding statement in a similar manner.  
Lastly, for $\alpha, \beta \leq 1$, we combine the above cases.
\end{proof}

Note that every polynomial, in particular the $X_m$-Jacobi polynomials, will satisfy all of the boundary conditions given by \eqref{Boundary Conditions}.

Next, we adapt ideas introduced in \cite{KMUG3} and further developed in \cite{Liaw-Littlejohn-Milson-Stewart} (for the case of exceptional Laguerre orthogonal polynomial systems) to prove that the spectrum of the self-adjoint operators from Theorem \ref{t-SA-operator} consists exactly of the eigenvalues corresponding to the exceptional $X_m$-Jacobi polynomials (and nothing more).

\begin{theorem}\label{t-completeness}
The exceptional $X_m$-Jacobi polynomials $\left\{P_{m,n}^{(\alpha,\beta)}\right\}_{n=m}^\infty$ form a complete set of eigenfunctions of the self-adjoint operator $S_{\alpha,\beta,m}$ in $L^2((-1,1),W_{\alpha,\beta, m})$.  Additionally, the spectrum $\sigma(S_{\alpha,\beta,m})$ of $S_{\alpha,\beta, m}$ is pure discrete spectrum consisting of the simple eigenvalues
\[
\sigma(S_{\alpha, \beta,m})=\sigma_p(S_{\alpha, \beta, m})=\left\{-(n-m)(1+\alpha+\beta+n-m)\mid n\geq m\right\}\,.
\]
\end{theorem}

\begin{proof}
The eigenvalue equations follow by the Darboux relations. 

It remains to prove the completeness of $\left\{P_{m,n}^{(\alpha,\beta)}\right\}_{n=m}^\infty$ in $L^2((-1,1),W_{\alpha,\beta, m})$. Fix $\alpha, \beta$ in the allowed range, pick $f\in \mathcal{H} = L^2((-1,1),W_{\alpha,\beta, m})$ and choose $\varepsilon>0$.

Define the function
\[
\widetilde f(x) := \frac{f(x)}{P_m^{(-\alpha-1,\beta-1)}(x)}\,.
\]
From the relationship $$W_{\alpha,\beta,m}(x) = \frac{W_{\alpha,\beta}(x)}{\left(P_m^{(-\alpha-1,\beta-1)}(x)\right)^2}$$ 
between the exceptional and the classical weight ($W_{\alpha,\beta,m}$ and $W_{\alpha,\beta}$, respectively), it easily follows
$$\|f\|_\mathcal{H} = \|\widetilde f\|_{L^2((-1,1);W_{\alpha,\beta})}.$$
In particular, we have $\widetilde f\in L^2((-1,1);W_{\alpha,\beta})$. 

Next we apply Lemma \ref{l-eta} with the function
\[
\eta(x) = P_m^{(-\alpha-1,\beta-1)}(x)
\]
and obtain the existence of $p\in\mathcal{P}$ such that
\[
\left\|\widetilde f - P_m^{(-\alpha-1,\beta-1)}(x) p(x) \right\|_{L^2((-1,1);W_{\alpha,\beta})}
<\varepsilon^2.
\]
Let $N$ be the degree of $p$.
With this polynomial $p$ we can compute
\begin{align*}
\varepsilon^2
&>
\left\|\widetilde f - P_m^{(-\alpha-1,\beta-1)}(x) p(x) \right\|_{L^2((-1,1);W_{\alpha,\beta})}
 =
\left\|f - \left(P_m^{(-\alpha-1,\beta-1)}(x)\right)^2 p(x) \right\|_{\mathcal H} .
\end{align*}

Our goal is to show that the approximant $\left(P_m^{(-\alpha-1,\beta-1)}(x)\right)^2 p(x)$ is contained in the (closure of the) vector space spanned by the exceptional Jacobi polynomials.
To this end, we consider two $(n+m+1)$-dimensional vector spaces
\begin{align*}
\mathcal{E}_{n+2m}&:=\left\{P_{m,j}^{(\alpha, \beta)}:j=m, m+1, \hdots, n+2m \right\} \text{ and }\\
\mathcal{F}_{n+2m}&:=\left\{q\in \mathcal{P}_{n+2m}:(1+x_i)q'(x_i)+\beta q(x_i) = 0\right\},
\end{align*}
where we let $x_i$ denote the $m-1$ roots of the polynomial $P_m^{(-\alpha-1,\beta-1)}(x)$.
The space $\mathcal{F}_{n+2m}$ is motivated by the exceptional term in the exceptional Jacobi differential expression.
Clearly, we have $\left(P_m^{(-\alpha-1,\beta-1)}(x)\right)^2 p(x)\in \mathcal{F}_{n+2m}.$ Since $\text{dim}\,\mathcal{F}_{n+2m} = \text{dim}\,\mathcal{E}_{n+2m}$ we achieve our goal, if we can show that
\[
\mathcal{E}_{n+2m}\subset
\mathcal{F}_{n+2m}.
\]

Take $Q\in \mathcal{E}_{n+2m}$. Since $\mathcal{E}_{n+2m}$ is spanned by a basis of eigenvectors of the exceptional $X_m$-Jacobi differential expression $T_{\alpha,\beta,m}$ we have $T_{\alpha,\beta,m}[\mathcal{E}_{n+2m}]\subset \mathcal{E}_{n+2m}$. It follows that
\begin{align*}
T_{\alpha,\beta,m}[Q] &:=(1-x^2)Q''+(\beta-\alpha-(\beta+\alpha+2)x-2(1-x^2)\left(\log(P_m^{(-\alpha-1,\beta-1)})\right)'Q' \\
&\quad +(\alpha-\beta-m+1)m-2\beta(1-x)\left(\log(P_m^{(-\alpha-1,\beta-1)})\right)'Q
\end{align*}
is polynomial, and hence the exceptional term (that is, the only term with a denominator):
\[
-2(1-x) \frac{\left(P_m^{(-\alpha-1,\beta-1)}(x)\right)'}{P_m^{(-\alpha-1,\beta-1)}(x)}\left[(1+x)Q'(x)+\beta Q(x)\right]
\]
is polynomial. Since the roots of the classical orthogonal are simple and 1 is not a root, we have
\[
(1+x_i)Q'(x_i)+\beta Q(x_i) = 0.
\]
We obtain $Q\in \mathcal{F}_{n+2m}$ as desired.
\end{proof}

Let $\mathcal{P}$ denote the set of all polynomials.

\begin{lemma}\label{l-eta}
Given a function $\eta$ on $[-1,1]$ that satisfies $0<c<|\eta(x)|<C<\infty$ for all $x\in [-1,1]$. Then the set $\{\eta(x)p(x):p\in\mathcal{P}\}$ is dense in $L^2((-1,1);W_{\alpha,\beta})$ for the classical range of parameters $\alpha, \beta$, and the classical Jacobi weight $W_{\alpha,\beta}$.
\end{lemma}

\begin{proof}
Fix $\alpha, \beta$ in the classical parameter range; that is, $\alpha, \beta> -1$. Then, by the theory of classical orthogonal polynomials, the polynomials $\mathcal{P}$ are dense in $\mathcal{H} = L^2((-1,1); W_{\alpha,\beta})$. Therefore, it suffices to show that
\[
\mathcal{P}\subset \textrm{clos}_{\mathcal{H}}(\eta \mathcal{P}).
\]

To show this, take $p\in \mathcal{P}$ and fix $\varepsilon>0$. First observe that
\[
\|p/\eta\|_\mathcal{H}\le
(1/c) \|p\|_\mathcal{H},
\]
so that $p/\eta \in \mathcal{H}$.
By taking $q\in \mathcal{P}$ such that
\begin{align*}
\varepsilon^2
>
C^2 \|p/\eta - q\|_\mathcal{H}^2
\ge
\|(p/\eta - q)\eta\|_\mathcal{H}^2
=
\|p- \eta q\|_\mathcal{H}^2,
\end{align*}
the lemma is proved.
\end{proof}


\begin{thebibliography}{10}

\bibitem{Glazman}
N.~I. Akhiezer and I.~M. Glazman.
\newblock {\em Theory of linear operators in {H}ilbert space}.
\newblock Dover Publications, Inc., New York, 1993.
\newblock Translated from the Russian and with a preface by Merlynd Nestell,
  Reprint of the 1961 and 1963 translations, Two volumes bound as one.

\bibitem{Atia-Littlejohn-Stewart}
M.~J. Atia, L.~L. Littlejohn, and J.~Stewart.
\newblock Spectral theory of ${X}_1$-{L}aguerre polynomials.
\newblock {\em Adv. Dyn. Syst. Appl.}, 8(2):181--192, 2013.

\bibitem{Bochner}
S.~Bochner.
\newblock \"{U}ber {S}turm-{L}iouvillesche {P}olynomsysteme.
\newblock {\em Math. Z.}, 29(1):730--736, 1929.

\bibitem{KMUG8}
D.~G{\'o}mez-Ullate, N.~Kamran, and R.~Milson.
\newblock Quasi-exact solvability and the direct approach to invariant
  subspaces.
\newblock {\em J. Phys. A}, 38(9):2005--2019, 2005.

\bibitem{KMUG}
D.~G{\'o}mez-Ullate, N.~Kamran, and R.~Milson.
\newblock An extended class of orthogonal polynomials defined by a
  {S}turm-{L}iouville problem.
\newblock {\em J. Math. Anal. Appl.}, 359(1):352--367, 2009.

\bibitem{KMUG2}
D.~G{\'o}mez-Ullate, N.~Kamran, and R.~Milson.
\newblock Exceptional orthogonal polynomials and the {D}arboux transformation.
\newblock {\em J. Phys. A}, 43(43):434016, 16, 2010.

\bibitem{KMUG1}
D.~G{\'o}mez-Ullate, N.~Kamran, and R.~Milson.
\newblock An extension of {B}ochner's problem: exceptional invariant subspaces.
\newblock {\em J. Approx. Theory}, 162(5):987--1006, 2010.

\bibitem{KMUG3}
D.~G{\'o}mez-Ullate, N.~Kamran, and R.~Milson.
\newblock Two-step {D}arboux transformations and exceptional {L}aguerre
  polynomials.
\newblock {\em J. Math. Anal. Appl.}, 387(1):410--418, 2012.

\bibitem{KMUG5}
D.~G{\'o}mez-Ullate, N.~Kamran, and R.~Milson.
\newblock A conjecture on exceptional orthogonal polynomials.
\newblock {\em Found. Comput. Math.}, 13(4):615--666, 2013.

\bibitem{Zeros}
D.~G{\'o}mez-Ullate, F.~Marcell{\'a}n, and R.~Milson.
\newblock Asymptotic and interlacing properties of zeros of exceptional
  {J}acobi and {L}aguerre polynomials.
\newblock {\em J. Math. Anal. Appl.}, 399(2):480--495, 2013.

\bibitem{Gonzalez}
A.~Gonz{\'a}lez-L{\'o}pez, N.~Kamran, and P.~J. Olver.
\newblock Normalizability of one-dimensional quasi-exactly solvable
  {S}chr\"odinger operators.
\newblock {\em Comm. Math. Phys.}, 153(1):117--146, 1993.

\bibitem{Hellwig}
G.~Hellwig.
\newblock {\em Differential operators of mathematical physics. {A}n
  introduction}.
\newblock Translated from the German by Birgitta Hellwig. Addison-Wesley
  Publishing Co., Reading, Mass.-London-Don Mills, Ont., 1967.

\bibitem{Kamran-Olver}
N.~Kamran and P.~J. Olver.
\newblock Lie algebras of differential operators and {L}ie-algebraic
  potentials.
\newblock {\em J. Math. Anal. Appl.}, 145(2):342--356, 1990.

\bibitem{Liaw-Littlejohn-Milson-Stewart}
C.~Liaw, L.~L. Littlejohn, R.~Milson, and J.~Stewart.
\newblock A new class of exceptional orthogonal polynomials: The type III $X_m-$Laguerre polynomials and the spectral analysis of three types of exceptional Laguerre polynomials.
\newblock Submitted, see \href{http://arxiv.org/abs/1407.4145}{arXiv:1407.4145}.

\bibitem{Liaw-Littlejohn-Stewart}
C.~Liaw, L.~L. Littlejohn, J.~Stewart, and Q.~Wicks.
\newblock The spectral analysis of the {J}acobi expression for extreme
  parameter choices.
\newblock {\em J.~Math. Anal.~Appl.~}{\bf 422} (2014) 212--239, DOI: 10.1016/j.jmaa.2014.08.016.

\bibitem{Naimark}
M.~A. Naimark.
\newblock {\em Linear differential operators. {P}art {II}: {L}inear
  differential operators in {H}ilbert space}.
\newblock With additional material by the author, and a supplement by V. \`E.
  Ljance. Translated from the Russian by E. R. Dawson. English translation
  edited by W. N. Everitt. Frederick Ungar Publishing Co., New York, 1968.

\bibitem{Odake-Sasaki1}
S.~Odake and R.~Sasaki.
\newblock Infinitely many shape-invariant potentials and cubic identities of
  the {L}aguerre and {J}acobi polynomials.
\newblock {\em J. Math. Phys.}, 51(5):053513, 9, 2010.

\bibitem{Quesne}
C.~Quesne.
\newblock Exceptional orthogonal polynomials, exactly solvable potentials and
  supersymmetry.
\newblock {\em J. Phys. A}, 41(39):392001, 6, 2008.

\bibitem{Quesne2}
C.~Quesne.
\newblock Solvable rational potentials and exceptional orthogonal polynomials
  in supersymmetric quantum mechanics.
\newblock {\em SIGMA Symmetry Integrability Geom. Methods Appl.}, 5:Paper 084,
  24, 2009.

\bibitem{Routh}
E.~J. Routh.
\newblock On some properties of certain solutions of a differential equation of
  the second order.
\newblock {\em Proc. London Math. Soc.}, S1-16(1):245--262, 1885.

\end{thebibliography}
\end{document}